\documentclass[12pt]{article}
 \usepackage{amsmath}
\usepackage{amsfonts}
\usepackage{amssymb}
\newcommand{\cA}{{\mathfrak A}}

\newtheorem{proposition}{Proposition}[section]
\newtheorem{example}{Example}[section]

\newtheorem{definition}{Definition}[section]

\newcommand{\bgeqn}{\begin{eqnarray}}
\newcommand{\edeqn}{\end{eqnarray}}
\newcommand{\bgeq}{\begin{eqnarray*}}
\newcommand{\edeq}{\end{eqnarray*}}
\newcommand{\bec}{\begin{center}}
\newcommand{\enc}{\end{center}}

\def\Min{\mathop{\rm Min}}

\newcommand{\Z}{{\cal Z}}
\newcommand{\F}{{\cal F}}

\newcommand{\G}{{\cal G}}

\newcommand{\X}{{\cal X}}

\newcommand{\LL}{{\cal L}}

\newcommand{\fb}{\rule[-2pt]{4pt}{8pt}}

\newcommand{\be}{\begin{equation}}
\newcommand{\ee}{\end{equation}}

\def\w{\omega}

\def\O{\Omega}

\def\ess {{\rm ess}\sup}

\newcommand{\avr}{{\sf AV@R}}

\def\bbr{{\Bbb{R}}} 
\def\bbe{{\Bbb{E}}} 

\newcommand{\ind}{{\mbox{\boldmath$1$}}}

\begin{document}
\title{\bf  Decomposability and time consistency of risk averse  multistage programs }

\author{
{\bf A. Shapiro} \\
School of Industrial and Systems Engineering\\
Georgia Institute of Technology\\
Atlanta, GA 30332-0205\\
 \and
{\bf K. Ugurlu}\\
Department of Mathematics\\
University of Southern California\\
Los Angeles, CA 90089
}
\date{}
 \maketitle

 \thispagestyle{empty}

\begin{abstract}
  
  Two approaches to time consistency of risk averse multistage stochastic   problems were discussed in the recent literature. In one approach certain properties of the corresponding  risk measure  are postulated which imply its  decomposability. The other approach deals directly with conditional   optimality of solutions of the considered  problem. The aim of this paper is to discuss a relation between these two approaches.
  \end{abstract}


\setcounter{equation}{0}
\section{Introduction}
\label{sec:intr}

Consider the following risk averse  multistage stochastic  optimization problem
\begin{equation}
\label{mul1}
\begin{array}{cll}
\Min&\varrho\big [ f_1(x_{1}) + f_2(
x_{2},\w) +   \cdots +  f_T(x_{T},\w)
\big],\\
{\rm s.t.} &
x_1\in \X_1,\; x_t\in \X_t(x_{t-1},\w),\;
t=2,...,T-1,
\end{array}
\end{equation}
where optimization is performed over policies $\{x_1,x_2(\w),...,x_T(\w)\}$ adapted to a filtration
$\F_1\subset \F_2\subset \cdots\subset \F_T$  and   $\varrho(Z)$ is a  risk measure  (cf.,   \cite[Section 6.8.5]{SDR}). In particular if $\varrho$ is the expectation operator, then this becomes the standard  risk neutral formulation.

Basically two approaches to time consistency of risk averse multistage   problems were discussed in the recent literature. In one approach certain properties of risk measure $\varrho$ are postulated which imply  decomposability of $\varrho$ and hence possibility of writing problem  (\ref{mul1}) in a nested form similar to the risk neutral case (cf., \cite{rusz10} and references therein). The other approach deals directly with sequential conditional   optimality of solutions of problem (\ref{mul1}). This  is related to the so-called
Bellman's Principle of Optimality,  \cite{bell57}:
 {\em ``An optimal policy has the property that whatever the initial state and initial decision are, the remaining decisions must constitute an optimal policy with regard to the state resulting from the first decision"}.
In a slightly different form this principle had been formulated in \cite{carp} as: {\em
``The decision maker formulates
an optimization problem at time $t_0$ that yields a sequence of optimal decision
rules for $t_0$ and for the following time steps $t_1,...,t_N=T$. Then, at the next time
step $t_1$, he formulates a new problem starting at $t_1$ that yields a new sequence of
optimal decision rules from time steps $t_1$ to $T$. Suppose the process continues until
time $T$ is reached. The sequence of optimization problems is said to be dynamically
consistent if the optimal strategies obtained when solving the original problem at
time $t_0$ remain optimal for all subsequent problems."}
From a conceptual point of view this is a natural principle - an optimal solution obtained by solving the problem at the first stage remains optimal from the point of view of later stages.

A natural question is whether the  decomposability of $\varrho$  implies  dynamical consistency of optimal solutions of problem  (\ref{mul1}). It is shown in \cite[Proposition 6.80]{SDR} that indeed such implication holds if problem  (\ref{mul1}) has unique optimal solution or in case of multiple solutions  under a stronger notion of  strict monotonicity of the involved  nested risk measures. The aim of this paper is to make a further investigation of relations between these two approaches to time consistency. In particular we construct an example showing that in absence of the strict monotonicity condition,   the decomposability of $\varrho$ does not necessarily imply the dynamical consistency when problem   (\ref{mul1}) possesses    several optimal solutions.

\setcounter{equation}{0}
\section{Time consistency of risk measures}
\label{sec:tc}

In this section we overview some results on  risk measures and decomposability of the corresponding risk averse problems. We follow    \cite[Section 6.8]{SDR} and use the following framework. Let $(\O,\F,P)$ be a probability space and $\F_1\subset \F_2\subset \cdots\subset \F_T$ be a sequence of sigma algebras (a filtration)  with sigma algebra  $\F_1=\{\emptyset,\O\}$ being trivial  and $\F_T=\F$. For $p\in [1,\infty]$ consider the spaces $\Z_t:=L_p(\O,\F_t,P)$ of $\F_t$-measurable  $p$-integrable functions (random variables)  $Z:\O\to\bbr$. Note that  $\Z_1\subset \cdots\subset \Z_T$, $\Z_T=L_p(\O,\F,P)$, and since $\F_1$ is trivial  the space $\Z_1$ consists of constant on $\O$ functions and can be identified with $\bbr$.

It is said that risk measure $\varrho:\Z_T\to\bbr$ is {\em decomposable} if it can be represented as the composition  $\varrho =\rho_2\circ \cdots \circ \rho_T$ of coherent  conditional risk mappings $\rho_{t}:\Z_{t}\to \Z_{t-1}$, $t=2,...,T$. In particular for $Z=Z_1+...+ Z_T$, $Z_t\in \Z_t$, we have then
\begin{equation}\label{ris1}
 \varrho(Z)=Z_1+\rho_2\left(Z_2+\cdots+\rho_{T-1}
 \big(Z_{T-1}+\rho_T(Z_T)\big)\right).
\end{equation}
An example of decomposable risk measure is the expectation operator. That is,
\begin{equation}\label{ris2}
 \bbe[Z]=\bbe_{|\F_1}\big(\cdots \bbe_{|\F_{T-1}}(Z)\big),
\end{equation}
with  $\rho_t$ is  given by the conditional expectation $\bbe_{|\F_{t-1}}$. Note that $\bbe_{|\F_1}=\bbe$ since $\F_1=\{\emptyset,\O\}$.
Another example of the decomposable risk measure is the essential supremum operator $\varrho(Z)=\ess (Z)$, $Z\in L_\infty (\O,\F,P)$, with the corresponding mappings $\rho_t$ are given by the respective conditional essential supremum operators.

Recall that  mapping  $\rho_{t}:\Z_{t}\to \Z_{t-1}$ is said to be a coherent conditional risk mapping if it satisfies the following conditions (R1)--(R4)
(for real valued risk measures these conditions were introduced in the pioneering paper \cite{ADEH:1999}).
The notation $Z  \succeq  Z'$ means that $Z(\w)\ge Z'(\w)$ for a.e. $\w\in \O$. We also use notation $Z  \succ   Z'$ meaning that $Z  \succeq  Z'$ and $Z\ne Z'$, that is  $Z  \succeq  Z'$ and  $Z(\w)>Z'(\w)$ on a set of positive probability.
\begin{description}
\item[(R1)] {\sl Convexity}:
\[
\rho_{t}(\alpha Z  + (1-\alpha)Z') \preceq
\alpha\rho_{t}(Z) + (1-\alpha)\rho_{t}(Z'),
\]
for any $Z,Z'\in\Z_{t}$ and $\alpha\in [0,1]$.
 \item[(R2)] {\sl Monotonicity}: If
$Z,Z'\in\Z_{t} $ and  $Z  \succeq  Z'$, then
$\rho_{t}(Z)\succeq \rho_{t}(Z')$.
\item[(R3)] {\sl Translation Equivariance}: If
$Y\in \Z_{t-1} $ and $Z\in \Z_{t} $, then
$\rho_{t}(Z+Y)=\rho_{t}(Z)+Y.$
\item[(R4)]
{\sl Positive homogeneity}: If $\alpha \geq 0$ and
$Z\in \Z_{t}$, then $\rho_{t}(\alpha
Z)=\alpha\rho_{t}(Z)$.
\end{description}
It is shown in \cite{rusz10} that some natural conditions (axioms) necessarily imply the decomposability of risk measure $\varrho$.

In the   definition below  of time (dynamical) consistency of optimal policies we follow
\cite[Definition 6.81]{SDR}.

\begin{definition}
For a decomposable risk measure $\varrho:\Z_T\to\bbr$,
we say that an optimal policy $\{\bar{x}_1,\bar{x}_2(\w),...,\bar{x}_T(\w)\}$ of problem {\rm (\ref{mul1})} is {\em time consistent} if the policy $\{\bar{x}_t(\w),...,\bar{x}_T(\w)\}$
 is optimal for problem
\begin{equation}
\label{ris3}
\begin{array}{cll}
\Min&\varrho_{t,T}\big [ f_t(x_{t},\w) +
     \cdots +  f_T(x_{T},\w)
\big],\\
{\rm s.t.} &
 x_\tau\in \X_\tau(x_{\tau-1},\w),\;
\tau=t,...,T,
\end{array}
\end{equation}
 conditional on $\F_t$ and $\bar{x}_{t-1}$,   $t=2,...,T$,
 where $\varrho_{t,T}:=\rho_t\circ \cdots \circ \rho_T:\Z_T\to\Z_{t-1}$.
 \end{definition}

It was believed that  the implication:
 \begin{equation}
\label{ris4}
{\rm
  ``decomposability\; of\;  \varrho"}\;  \Rightarrow \;  {\rm ``time\; consistency\; of\; optimal \;solutions"}
  \end{equation}
holds for any decomposable risk measure $\varrho$.
It is shown in \cite[Proposition 6.80]{SDR} that indeed if problem (\ref{mul1}) has a {\em unique} optimal solution, then the implication (\ref{ris4}) follows. However, when  problem (\ref{mul1}) has more than one optimal solution, in order to ensure
 the implication (\ref{ris4}) a stronger notion of strict monotonicity was needed in the proof.
\begin{description}
 \item[(R$'$2)] {\sl Strict Monotonicity}: If
$Z,Z'\in\Z_{t} $ and  $Z  \succ   Z'$, then
$\rho_{t}(Z) \succ  \rho_{t}(Z')$.
\end{description}
In the example below we demonstrate that indeed  without strict monotonicity the implication (\ref{ris4}) may fail.

\begin{example}
\label{ex-1}
{\rm
Consider   the following settings. Number of stages  $T=3$, the underlying probability space is finite and  defined by the following  scenario tree. At the first stage there is one root node $\w_1$, at the  second stage there are two nodes $\w_2^1$ and $\w_2^2$ with respective probabilities 1/2 of moving to these nodes; there are two branches for each node at stage 2 with 4  nodes at stage three with nodes   $\w_3^1$ and $\w_3^2$ denoting children nodes of node $\w_2^1$ and  nodes   $\w_3^3$ and $\w_3^4$ denoting children nodes of node $\w_2^2$,  and  respective conditional probabilities of 1/2. So the total number of scenarios is 4 each with equal probability 1/4. The decision variables are one dimensional and cost functions are linear, i.e.,
$f_1(x_1):=c_1 x_1$, $f_2(x_{2},\w_2^1):=c_2^1 x_2$, $f_2(x_{2},\w_2^2):=c_2^2 x_2$,
$f_3(x_{3},\w_2^j):=c_2^j x_3$, $j=1,...,4$. Feasible sets are $\X_1:=\{0\}$, $\X_2(x_1,\w_2^1)=\X_2(x_1,\w_2^2):=\{0\}$, $\X_3(x_2,\w_3^j):=[1,2]$, $j=1,...,4$.
The risk measure $\varrho$ is taken to be the max-operator, i.e., $\varrho(Z):=\max_{\w\in \O}Z(\w)$.

 As it was pointed above this risk measure is decomposable (since  the space $\O$ is finite the max-operator is the same  as the $`\ess$' operator). However, it is not difficult to see that the max-operator does not satisfy the strict monotonicity condition (R$'$2).

Note that here we need to consider decision $x_3=x_3(\w_3)$ as a function of $\w_3$ only.
  Hence problem (\ref{mul1}) takes the form of finding   policy
$x_1,x_2(\w_2^i),x_3(\w_3^j)$ which solves the minimax problem
\begin{equation}
\label{mul2}
\Min \max_\w\big \{c_1x_1+c_2^i x_2(\w_2^i)+c_3^j x_3(\w_3^j)
\big\},
\end{equation}
over feasible policies.
Consider the following   coefficients $c_1:=0$,   $c_2^1=c_2^2:=0$, $c_3^1=c_3^2:=1$, $c_3^3=c_3^4:=4$.
 Clearly  the following  policy $x_1=0$, $x_2(\w_2^1)=x_2(\w_2^2)=0$,
$x_3(\w_3^j)=1$, $j=1,...,4$, is an optimal solution of problem (\ref{mul2})
 with the corresponding optimal value 4. Also this policy is time consistent.

On the other hand consider    policy
$x_1=0$, $x_2(\w_2^1)=x_2(\w_2^2)=0$, $x_3(\w_3^1)=x_3(\w_3^2)=2$, $x_3(\w_3^3)=x_3(\w_3^4)=1$. This policy has also  value 4 and hence is optimal.
However, conditional on $\w_2=\w_2^1$ this policy has value 2, while the corresponding conditional optimal value  is 1. Hence conditional on $\w_2=\w_2^1$ this policy is not optimal, and thus is not time consistent.  \fb
}
\end{example}

Consider now the Average Value-at-Risk measure
\[
 \avr_\alpha(Z)=\inf_{t\in \bbr}\left\{t+\alpha^{-1}\bbe[Z-t]_+\right\},\;\alpha\in (0,1].
 \]
This is a coherent risk measure, but for $\alpha\in (0,1)$  it is not strictly monotone.
  Suppose that the risk measure $\varrho$ is given as nested $\avr_\alpha$ risk measure, e.g., for $T=3$ it is $\varrho(Z)=\avr_\alpha(\avr_{\alpha|\F_2}(Z))$. Note that for finite space $\O=\{\w_1,...,\w_m\}$, equipped with equal probabilities $1/m$, and for $\alpha<1/m$, it follows that $ \avr_\alpha(Z)=\max_{\w\in \O}Z(\w)$. Therefore for $\alpha\in (0,1/2)$,  Example \ref{ex-1} gives an example of a time inconsistent optimal policy for decomposable  nested $\avr$  risk measure.

\setcounter{equation}{0}
\section{Strict monotonicity}
\label{sec:str}

As it was pointed in the previous section, in order to ensure time consistency of optimal policies the stronger condition of strict monotonicity is needed. Let us first consider a real valued coherent risk measure $\rho:\Z\to \bbr$, where $\Z =L_p(\O,\F,P)$, $p\in [1,\infty)$.
 Consider the  dual space $\Z^*=L_q(\O,\F,P)$, $1/p+1/q=1$, $q\in (1,\infty]$.
 It follows that $\rho(\cdot)$ is continuous (in the norm topology of  $L_p(\O,\F,P)$) and has the following dual representation (cf., \cite{rs06})
\begin{equation}\label{mon1}
 \rho(Z)=\sup_{\zeta\in \cA}\int_\O \zeta(\w) Z(\w)dP(\w),
\end{equation}
where $\cA\subset \Z^*$ is a convex weakly$^*$ compact set of  density functions. Recall that the subdifferential of $\rho$ is given by
 \begin{equation}\label{mon2}
 \partial \rho(Z)=\arg\max_{\zeta\in \cA}\int_\O \zeta(\w) Z(\w)dP(\w).
\end{equation}
Note that since $\cA$ is   weakly$^*$ compact, the  set in the right hand side of (\ref{mon2}) is nonempty.

The following result is an extension of \cite[Proposition 6.38]{SDR}.

\begin{proposition}
\label{pr-coh}
Coherent risk measure $\rho:\Z\to\bbr$ satisfies the strict monotonicity condition iff the following condition holds:
\begin{equation}\label{mon3}
 \zeta(\w)>0\; {\rm for\;a.e.}\;\w\in \O,\;\forall Z\in \Z,\;\forall \zeta\in \partial \rho(Z).
\end{equation}
\end{proposition}

\begin{proof}
Suppose that condition (\ref{mon3}) holds. Let $Z\succ Z'$ and $\zeta\in \partial \rho(Z')$. Then by    (\ref{mon2}) we have
$\rho(Z')=\int  \zeta  Z' dP.$
Because of (\ref{mon1}) we also have that $\rho(Z)\ge \int  \zeta  Z dP$. Since $Z\succ Z'$ it follows by  condition (\ref{mon3})   that  $\int  \zeta  Z dP>\int  \zeta  Z' dP$, and hence $\rho(Z)>\rho (Z')$.

For the converse implication we argue by a contradiction. Suppose that there exist $\zeta\in \partial \rho(Z)$ and $A\in \F$ such that $P(A)>0$ and  $\zeta(\w)=0$ for all $\w\in A$. Consider $Z':=Z-\ind_A$. Clearly $Z\succ Z'$ and hence $\rho(Z)\ge \rho(Z')$. Moreover
\[
\rho(Z)=\int_\O  \zeta  Z  dP = \int_\O  \zeta  Z  dP-\int_A \zeta dP=
\int_\O  \zeta  Z' dP\le \rho(Z').
\]
It follows that $\rho(Z)=\rho(Z')$, a contradiction with strict monotonicity. \fb
\end{proof}
\\

Formulas for subdifferentials of various coherent risk measures can be found e.g. in \cite[Section 6.3.2]{SDR}.

For a random variable $Z:\O\to\bbr$ consider the corresponding cumulative distribution function (cdf) $F_Z(z)=P(Z\le z)$ and the (left side) quantile $F_Z^{-1}(t)=\inf\{z:F_Z(z)\ge t\}$.  It is said that risk measure $\rho:\Z\to\bbr$ is {\em law invariant} if for any $Z,Z'\in \Z$ having the same cdf it follows that $\rho(Z)=\rho(Z')$. Assume  that $\rho$ is law invariant and  either the space $(\O,\F,P)$ is nonatomic or $\O=\{\w_1,...,\w_m\}$ is finite  equipped with equal probabilities $1/m$. Then the dual representation (\ref{mon1}) can be written in the following form
\begin{equation}\label{mon4}
 \rho(Z)=\sup_{\sigma\in \Upsilon}\int_0^1 \sigma(t) F_Z^{-1}(t)dt,
\end{equation}
where $\Upsilon\subset \LL_q$ is a set of  spectral functions (e.g., \cite[Section 6.3.3]{SDR}).
A function $\sigma:[0,1)\to\bbr_+$ is said to be spectral if it is monotonically nondecreasing, right side continuous and
$\int_0^1 \sigma(t) dt=1$. By $\LL_p$ we denote the corresponding space of $p$-integrable functions defined on $\O=[0,1]$ equipped with its Borel sigma algebra and uniform probability distribution. We can assume that the set $\Upsilon$ is weakly$^*$ compact. In particular if the set $\Upsilon$ is a singleton, then the risk measure $\rho$ is said to be spectral. For example $\avr_\alpha$ is a spectral risk measure with the corresponding spectral function
$\sigma(\cdot)=\alpha^{-1}\ind_{[1-\alpha,1]}(\cdot)$.

By Proposition \ref{pr-coh} we have that law invariant risk measure $\rho$ satisfies the strict monotonicity condition iff every spectral function
$\sigma\in\arg\max_{\sigma\in \Upsilon}\int_0^1 \sigma(t) F_Z^{-1}(t)dt$ is strictly positive on the interval $(0,1)$. In particular, it follows that   the
$ \avr_\alpha$
 is not  strictly monotone for $\alpha\in (0,1)$.

Let $\G$ be a strict  subalgebra of $\F$, i.e. $\G\ne \F$,  and $\Z':=L_p(\O,\G,P)$. With a law invariant coherent  risk measure $\rho:\Z\to\bbr$ is associated coherent risk mapping  $\rho_{|\G}:\Z\to\Z'$
by replacing  $F^{-1}_Z$ in (\ref{mon4}) with its conditional counterpart
$F_{Z|\G}^{-1}$.
 We have then that  $ \rho_{|\G}$ is strictly monotone iff $\rho$ is strictly monotone. In particular the conditional $\avr_{\alpha|\G}$ is not strictly monotone for $\alpha\in (0,1)$.
 \\
 \\
{\bf Acknowledgments}\\
\\
 Research of the first  author  was partly supported by  DARPA EQUiPS program, grant SNL 014150709.


\begin{thebibliography}{99}



\bibitem{ADEH:1999}
  Artzner, P.,    Delbaen, F.,    Eber,  J.-M.
 and   Heath, D.,
  Coherent measures of risk,
  {\em Mathematical Finance},   9 (1999),  203--228.

\bibitem{bell57}
Bellman, R.E. (1957). {\em Dynamic Programming}. Princeton University Press, Princeton, NJ.

\bibitem{carp}
P. Carpentier, J.P. Chancelier, G. Cohen, M. De Lara and P. Girardeau, Dynamic consistency for stochastic optimal control problems, {\em Annals of Operations Research}, 200 (2012), 247-263.

\bibitem{rs06}
A.  Ruszczy\'{n}ski and  A. Shapiro,
  Optimization of convex risk functions,
  {\em Mathematics of Operations Research},   31 (2006), 433--452.

\bibitem
{rusz10}
Ruszczy\'{n}ski, A.,
Risk-averse dynamic programming for Markov decision processes. {\em Mathematical Programming, Series B}, 125 (2010), 235--261.


\bibitem{SDR}
  Shapiro, A., Dentcheva, D.   and    Ruszczy\'{n}ski, A.,  {\em Lectures on Stochastic Programming: Modeling and Theory}, second edition,   SIAM, Philadelphia, 2014.

\end{thebibliography}
\end{document}